\documentclass[graybox,envcountsame,envcountresetsect]{svmult-fgs7}
\usepackage{titlesec}

\usepackage{type1cm}        
\usepackage{makeidx}         
\usepackage{graphicx}        
\usepackage{multicol}        
\usepackage[bottom]{footmisc}

\usepackage{newtxtext}       %
\usepackage[varvw]{newtxmath}       

\makeindex             

\begin{document}

\title*{Recent progress on pointwise normality of self-similar measures}
\author{Amir Algom}
\institute{Amir Algom \at Department of Mathematics, University of Haifa at Oranim, Tivon 36006, Israel \email{amir.algom@math.haifa.ac.il}
}
%
%
\maketitle

\abstract{This article is an exposition of recent results and methods on the prevalence of normal numbers in the support of self-similar measures on the line. We also provide an essentially self-contained proof of a recent Theorem that the Rajchman property (decay of the Fourier transform) implies that typical elements in the support of the measure are normal to all bases; as no decay rate is required, this improves  the classical criterion of Davenport,  Erdős, and LeVeque (1964). Open problems regarding effective equidistribution,  non-integer bases, and higher order correlations, are discussed.}

\numberwithin{theorem}{section}
\section{Introduction}
\label{Sec intro}
Let $b$ be an integer greater or equal to $2$. Let $T_b$ be the times $b$-map,
$$T_b(x)=b\cdot x \mod 1, x\in \mathbb{R}.$$
A number $x\in \mathbb{R}$ is called $b$-normal, or normal to base $b$, if its orbit $\lbrace T_b ^n (x) \rbrace_{n\in \mathbb{N}}$ equidistributes for $\lambda$, the Lebesgue measure on $[0,1]$. By a standard abuse of notation, $\lambda$ will also denote the Lebesgue measure on $\mathbb{T}:=\mathbb{R}/Z \simeq [0,1)$. That is, $x$ is normal to base $b$ if in the weak$^*$ sense
$$\lim_{N\rightarrow \infty} \frac{1}{N} \sum_{n=1} ^N \delta_{T_p ^n (x)} = \lambda.$$
In 1909 Borel proved that Lebesgue almost every $x$ is \textit{absolutely normal}, i.e., normal in all bases.  

 A highly active strand of research in metric number theory concerns finding optimal conditions on fractal measures and subsets of $\mathbb{R}$ under which they inherit the number-theoretic properties of the ambient space. In this spirit,  it is  believed that absolute normality should remain true for typical elements of well structured sets with respect to appropriate measures, absent obvious obstructions. The purpose of this paper is to discuss some of the recent progress made in this direction, to provide an exposition to  newly developed methods and tools, and to discuss some remaining open problems.

We will work with self-similar sets and measures on $\mathbb{R}$, which are defined as follows. Let $\Phi = \lbrace f_1,...f_n\rbrace\subseteq \text{Sim}(\mathbb{R})$ be a finite collection of invertible and contracting similarity maps, that leave some compact interval $J\subseteq \mathbb{R}$ invariant. That is, for all $f_i \in \Phi$,
$$f_i(x)=s_i\cdot x+t_i, \, s_i\in (-1,1)\setminus \lbrace 0 \rbrace,\, t_i\in \mathbb{R},\, f_i(J)\subseteq J.$$
We will refer to $\Phi$ as a \textit{self similar IFS} (Iterated Function System). 
It is well known that there exists an unique compact set $\emptyset \neq K=K_\Phi \subseteq J$ such that
$$K = \bigcup_{i=1} ^n f_i(K).$$
The set $K$ is called a \textit{self-similar set}, and the \textit{attractor} of the IFS $\lbrace f_1,...,f_n \rbrace$. We always assume that the fixed points of $f_i$ and $f_j$ differ for some $i,j$; this is known to ensure that $K$ is infinite.

Next, let $\mathbf{p}=\left( p_1,...,p_n \right)$ be  a strictly positive probability vector:
$$\sum_{i=1} ^n p_i =1,\, \text{ and } p_i>0 \text{ for all } i.$$
It is well known that there exists a unique probability measure $\nu=\nu_\mathbf{p}$ such that
$$\nu = \sum_{i=1} ^n p_i \cdot f_i\nu$$
where $f_i \nu$ is the push-forward of $\nu$ by $f_i$.  The measure $\nu$ is called a \textit{self-similar measure}, and is supported on $K$. The assumptions that $K$ is infinite and that each $p_i>0$ are known to ensure that $\nu$ is non-atomic. In particular, all self-similar measures in this note are non-atomic.

Let us now return to  the problem of absolute normality in the fractal setting. Given a measure $\mu$ on $\mathbb{R}$, we say it is \textit{pointwise absolutely normal} if $\mu$ almost every $x$ is absolutely normal. Now, it is not hard to see that for every compact interval $J$, the Lebesgue measure restricted to $J$ is a self-similar measure, whence it is pointwise absolutely normal by Borel's Theorem. On the other hand, it is also not hard to see that there exist self-similar measures that are not pointwise absolutely normal. For example, letting $\Phi = \lbrace f_1(x)=\frac{x}{3}, \, f_2(x)=\frac{x+2}{3} \rbrace$ with $\mathbf{p}=(\frac{1}{2},\frac{1}{2})$, we see that $\nu_\mathbf{p}$ is the Cantor-Lebesgue measure on the middle-$\frac{1}{3}$ Cantor set $K_\Phi$. Since no element in $K_\Phi$ can be normal to base $3$, $\nu_\mathbf{p}$ is not pointwise absolutely normal. The following Theorem shows that such digit restrictions are essentially the only possible obstruction to pointwise absolutely normality for self-similar measures: 

\begin{theorem} \label{Theorem main}
Let $\nu$ be a non-atomic self-similar measure on $\mathbb{R}$ with respect to an IFS $\Phi$. If $\nu$ is not absolutely pointwise normal then there exists an invertible affine map $g:\mathbb{R}\rightarrow \mathbb{R}$ such that the   $g$-conjugated  IFS $\Psi = \lbrace g^{-1} \circ f_i \circ g\rbrace_{f_i \in \Phi}$ has the following form.\newline
There exists some integer $b>1$ such that for every $h_i(x)=s_i\cdot x+t_i \in \Psi$:
\begin{enumerate}
\item We have $\frac{\log |s_i|}{\log b} \in \mathbb{Q};$ and,

\item There exist $k_i \in \mathbb{Z}$ and $q_i \in \mathbb{Q}_+$ such that $t_i = \frac{k_i}{b^{q_i}}$.
\end{enumerate}
\end{theorem}
Theorem \ref{Theorem main} is a  consequence of the recent work of the author, Rodriguez Hertz, and Wang \cite{algom2020decay}, combined with the work of the author, Baker, and Shmerkin \cite{Algom2022Baker}, or with the work of   B{\'a}r{\'a}ny,  K{\"a}enm{\"a}ki,  Py{\"o}r{\"a}l{\"a}, and  Wu \cite{barany2023scaling}. These works rely  on the recent breakthrough works of Hochman-Shmerkin \cite{hochmanshmerkin2015}, and of Hochman \cite{Hochman2020Host}, that provided new mechanisms to prove pointwise normality, that go beyond the classical approach of Davenport, Erdős, and LeVeque \cite{Davenport1964Erdos}. It also relies on recent progress on the Fourier decay problem for self-similar measures, due to Li-Sahlsten \cite{li2019trigonometric},  Br\'emont \cite{bremont2019rajchman}, and  Varj\'u-Yu \cite{varju2020fourier}.

This will be explained in greater detail in the next Section, where we outline and discuss  two main technical Theorems that imply Theorem \ref{Theorem main}. The proof of one them, Theorem \ref{main tech theorem 1}, will be given in later Sections  with full details. Finally, we note that using the recent work of Dayan,  Ganguly,  and Weiss \cite{dayan2020random}, it is possible, at least in some cases, to say something about the $q_i \in \mathbb{Q}_+$ and the structure of the affine conjugating map $g$ as in Theorem \ref{Theorem main} - see  the next Section for more details.


\subsection{Main technical Theorems, and deriving  Theorem \ref{Theorem main} from them}
For a measure $\mu \in \mathcal{P}(\mathbb{R})$, where $\mathcal{P}(X)$ is the space of Borel probability measures on a  metric space $X$, we denote its Fourier transform by
$$\mathcal{F}_q (\mu):=\int e^{2\pi i x q} \,d\mu (x), \text{ where } q\in \mathbb{R}.$$
The following Theorem was originally proved in \cite[Theorem 1.3]{algom2020decay}. Its proof follows by combining bits and pieces from  more general arguments about $C^{1+\gamma}$ IFSs given in \cite{algom2020decay}. Thus, we take this opportunity and provide  a self-contained proof with full details in Sections \ref{sec:The martingale arguement} and \ref{Section proof}  below. \newpage
\begin{theorem} \label{main tech theorem 1}
Let $\nu \in \mathcal{P}(\mathbb{R})$ be a non-atomic self-similar measure. If
$$\lim_{|q|\rightarrow \infty}  \mathcal{F}_q(\nu) =0$$
then $\nu$ is pointwise absolutely normal.
\end{theorem}
Theorem \ref{main tech theorem 1} is related to the classical approach of Davenport,  Erdős, and LeVeque \cite{Davenport1964Erdos} to finding normal numbers in the support of fractals measures $\mu \in \mathcal{P}(\mathbb{R})$. By integrating Weyl's criteria, they showed that sufficiently fast decay of the $L^2 (\mu)$ norms of certain trigonometric polynomials as in Weyl's criterion will lead to pointwise normality of $\mu$ in a given base $b$. It ensures, for example, that if for some $\alpha=\alpha(\mu)>0$ we have
$$ \mathcal{F}_q(\mu) = O \left( \frac{1}{|\log \log |q||^{1+\alpha} } \right), \text{ as } |q|\rightarrow \infty$$
then $\mu$ is pointwise absolutely normal. We note, however, the such bounds are usually hard to obtain (if true at all) in concrete situation, even for well structured measures $\mu$. Until fairly recently, this was in fact the only general  method of proving pointwise absolute normality, and was thus applied in a variety of setups.  For an  exposition to this method and its many applications, we refer to Bugeaud's book \cite[Chapters 1.2 and 4]{Bugeaud2012book}.

Theorem \ref{main tech theorem 1} is thus a dramatic improvement of this classical criterion for self-similar measures, as it merely requires the Rajchman property (vanishing of the Fourier transform at infinity). It is also related to a classical problem  posed by Kahane and Salem \cite{Kahane1964Salem}:
\begin{equation} \label{Eq question kahane}
\text{If } \mu\in \mathcal{P}(\mathbb{R}) \text{ is a Rajchman measure, is } \mu \text{ pointwise absolutely normal }?
\end{equation} 
A negative answer to question \eqref{Eq question kahane} was given by Lyons in 1986 \cite{Lyons1986Salem}. However, Theorem \ref{main tech theorem 1} shows that the answer to \eqref{Eq question kahane} is positive for the class of self-similar measures. For further refinements of Lyon's results, and a more general answer to question \eqref{Eq question kahane}, we refer to the recent works of Pramanik and Zhang \cite{Pra2024zhang1, Pra2024zhang2}.

Let us now proceed to discuss the second main ingredient in the proof of Theorem \ref{Theorem main}. It is concerned with the question of when is a self-similar measure $\mu\in \mathcal{P}(\mathbb{R})$ pointwise $b$-normal, for some integer $b>1$. This means that $\mu$-a.e. $x$ is $b$-normal, but $\mu$ is not required to be \textit{absolutely} pointwise normal. Here, it will be convenient to introduce some new notation: We write $a \not \sim b$ to indicate that $\frac{\log |a|}{\log |b|} \notin \mathbb{Q}$, and write $a\sim b$ otherwise. The following Theorem is proved in  \cite[Theorem 1.1]{Algom2022Baker}, and in \cite[Theorem 1.4]{barany2023scaling} using a somewhat different method, as we discuss below:
\begin{theorem} \label{Main technical Theorem 2}
Let $b>1$ be an integer, and let $\nu \in \mathcal{P}(\mathbb{R})$ be a non-atomic self-similar measure with respect to an IFS $\Phi$. If there exists $f_i \in \Phi$ such that $f_i ' \not \sim b$, then $\nu$ is pointwise $b$-normal. 
\end{theorem}
Theorem \ref{Main technical Theorem 2} has a long history: The first result in this direction was given by Cassels \cite{Cassels1960normal} and Schmidt \cite{Schmidt1960normal} independently around 1960, who showed that if $\mu$ is the Cantor-Lebesgue measure on the middle-$\frac{1}{3}$ Cantor set, then $\mu$ is pointwise $b$-normal for all $b \not\sim 3$. This was later generalized by Feldman and Smorodinsky \cite{Feldman1992normal} to all non-degenerate Cantor-Lebesgue measures with respect to any base $b$. An influential paper  of Host \cite{host1995normal} gave corresponding results about $T_p$-invariant measures, and related this line of research to  Furstenberg's $\times 2, \times 3$ Conjecture (see  also Meiri \cite{Meiri1998host}, and Lindenstrauss \cite{Elon2001host}).

In 2015 Hochman and Shmerkin \cite[Theorem 1.4]{hochmanshmerkin2015} proved Theorem \ref{Main technical Theorem 2} under the additional assumptions that $f_i' >0$ for all $i$, and that the intervals $f_i(J), 1\leq i \leq n,$ are disjoint except for potentially at their endpoints (a condition that is formally stronger than the well known open set condition). It actually followed from a more general result, \cite[Theorem 1.2]{hochmanshmerkin2015}, which provides a general criterion for a uniformly scaling measure to be pointwise $b$-normal. Informally, a uniformly scaling measure on $\mathbb{R}$ is a measure such  that for $\mu$-a.e. $x$, the family of measures one sees as one ``zooms'' into $x$, all the while re-scaling and conditioning the measure, equidistributes towards some distribution $P$ (supported on Borel probability measures). See e.g. \cite{Algom2022Baker} for more details. They also provided a general method of relating the distribution of $T_p$-orbits to the original measure - an idea that plays a key role in the proofs  of both Theorem \ref{Main technical Theorem 2} and Theorem \ref{main tech theorem 1}, see Section \ref{sec:The martingale arguement} below for more details. Now, up until very recently,  only self-similar measures with separation (say, with the open set condition, similarly to the condition stated above)   were known to be uniformly scaling. The main innovation introduced in \cite[Theorem 1.1]{Algom2022Baker} was a disintegration method that can be used to reduce to such a situation. However, a recent breakthrough of   B{\'a}r{\'a}ny,  K{\"a}enm{\"a}ki,  Py{\"o}r{\"a}l{\"a}, and  Wu \cite{barany2023scaling} shows that in fact all self-similar measures are uniformly scaling (also in higher dimensions, and for some smooth IFSs) - which thus also implies Theorem \ref{Main technical Theorem 2}, in a more direct way (see also \cite{Aleksi2021flow}).

Let us now explain how Theorems \ref{main tech theorem 1} and \ref{Main technical Theorem 2} can be used to deduce Theorem \ref{Theorem main}:
$$ $$
\noindent{ \textbf{Proof of Theorem \ref{Theorem main}}} Let $\nu \in \mathcal{P}(\mathbb{R})$ be a self-similar  measure with respect to an IFS $\Phi$, and assume $\nu$ is not pointwise absolutely normal. By Theorem \ref{main tech theorem 1}  $\nu$ cannot be a Rajchman measure. Thus, applying the recent works of Li-Sahlsten \cite{li2019trigonometric} and Br\'emont \cite{bremont2019rajchman} (see also Varj\'u-Yu \cite{varju2020fourier}) that provide a complete classification of which self-similar measures have the Rajchman property, we can deduce that $\Phi$ must have the following structure:

There exists an invertible affine map $g:\mathbb{R}\rightarrow \mathbb{R}$ such that the   $g$-conjugated  IFS $\Psi = \lbrace g^{-1} \circ f_i \circ g\rbrace_{f_i \in \Phi}$ satisfies:
\begin{enumerate}
\item There exists a Pisot number $\beta>1$ such that for every $g_i \in \Psi$, $\frac{\log |g_i'|}{\log \beta} \in \mathbb{Q}.$

Recall that $\beta>1$ is called a \textit{Pisot} number if it is an algebraic integer whose algebraic conjugates are all of modulus strictly less than $1$. Note that every integer larger than $1$ is a Pisot number. 

\item For every $g_i \in \Psi$ the translate $g_i(0) \in \mathbb{Q}(\beta)$, and in fact, has a more specialised algebraic structure (see \cite[Definitions 2.1 and 2.2]{bremont2019rajchman}).
\end{enumerate}
Finally, if for every integer $b>1$ we would have $b \not \sim \beta$, then by Theorem \ref{Main technical Theorem 2}  $\nu$ would be pointwise absolutely normal, which contradicts our assumptions. We can thus conclude that  $\beta \sim b$ for some integer $b>1$. In this case, opening up \cite[Definitions 2.1 and 2.2]{bremont2019rajchman} we see that in fact for every $g_i \in \Psi$ there exist $q_i \in \mathbb{Q}_+, k_i \in \mathbb{Z}$ such that $g_i(0) = \frac{k_i}{b^{q_i}}$. This completes the proof of Theorem \ref{Theorem main}. \hfill{$\Box$}
$$ $$
We remark that Dayan,  Ganguly,  and Weiss \cite{dayan2020random} recently gave sufficient conditions in terms of the translates of the IFS, that ensure pointwise normality in a given base. Using this result, it is possible to at least sometimes say something more about the conjugating map $g$, or to see that $q_i \in \mathbb{N}$. For example, it may be used to show that there are instances when $g' \in \mathbb{Q}$  - see e.g. \cite[Section 7.3]{algom2020decay}.

The rest of this paper is organized as follows: In Section \ref{sec:The martingale arguement} we provide an exposition to an important method developed by Hochman-Shmerkin \cite{hochmanshmerkin2015} and  Hochman \cite{Hochman2020Host}, to study equidistribution problems. This is Theorem \ref{Theorem Martingale}, whose proof involves the only non-elementary tool used to prove Theorem \ref{main tech theorem 1}  -  the ergodic Theorem for Martingale differences. Next, we proceed to prove Theorem \ref{main tech theorem 1} in Section \ref{Section proof}. One of the main innovations here is in the application of Theorem \ref{Theorem Martingale}, in such a way that does not require separation from the IFS.

We end the paper with a list of some remaining open problems about effective equidistribution, the case of non-integer bases, and higher order correlations. These are discussed  in detail in Section \ref{Section problems}.

\subsection{Acknowledgements}
The author thanks Meng Wu for his comments, in particular regarding Problem \ref{Question rajchman} and the non-Pisot case. We also thank Barak Weiss and Yann Bugeaud  for their remarks. The author was supported by Grant No. 2022034 from the United States - Israel Binational Science Foundation  (BSF), Jerusalem, Israel.

\section{Relating the distribution of orbits to the measure}
\label{sec:The martingale arguement}
In this Section we discuss a key idea towards Theorem \ref{main tech theorem 1}. Let $X$ be a compact metric space, and let $\mu \in \mathcal{P}(X)$ be a Borel probability measure. Let $T:X\rightarrow X$ be a measurable transformation. Let $\mathcal{C}_n$ denote a sequence (Borel) measurable partitions.  The following Theorem proves the following  useful fact: In order to understand  the statistics of the orbit of a $\mu$-typical point $x$,  one can instead consider the statistics of averages of the measures
$$T^n \mu_{\mathcal{C}_n(x)},$$
where $\mathcal{C}_n(x)$ is the unique $\mathcal{C}$-cell that contains $x$, and $\mu_{\mathcal{C}_n(x)}$ is the corresponding $\mu$-conditional measure (normalized restriction of $\mu$ to $\mathcal{C}_n(x)$, assuming the latter has positive $\mu$-mass). Of course, one needs to impose an assumption that will allow the $T$-dynamics to communicate with these partitions.

The first version of the following Theorem was given by   Hochman and Shmerkin \cite{hochmanshmerkin2015}; The version given here is due to Hochman \cite{Hochman2020Host}: 
\begin{theorem} \label{Theorem Martingale}
Let $X$ be a compact metric space and let $T\in \mathcal{C}(X)$. Let $\mathcal{A}_1,\mathcal{A}_2,..$ be a refining sequence of Borel partitions. Let $\mu \in\mathcal{P}(X)$. Suppose that
\begin{equation} \label{Eq 2.2} 
\lim_{k\rightarrow \infty} \sup_{n\in \mathbb{N}}\left \lbrace \text{diam}(T^n A):\, A\in \mathcal{A}_{n+k},\, \mu(A)>0 \right\rbrace = 0.
\end{equation}
Then for $\mu$-a.e. $x$, in the weak$^*$ sense we have
$$\lim_{N \rightarrow \infty} \left( \frac{1}{N} \sum_{n=1} ^N \delta_{T ^n (x)} \right) - \left( \frac{1}{N} \sum_{n=1} ^N T ^n \mu_{\mathcal{A}_n(x)} \right)=0.$$
\end{theorem}
The proof of Theorem \ref{Theorem Martingale} relies  on the ergodic Theorem for Martingale differences:
\begin{theorem} \label{Theorem erogidc diff}
Let $(X,\mathcal{B},\mu)$ be a probability space, and let
$$\mathcal{B}_1 \subseteq \mathcal{B}_2 \subseteq \mathcal{B}_3... \subseteq \mathcal{B}$$
be an increasing sequence of $\sigma$-algebras. Let $k\in \mathbb{N}$ and let $f_n \in L_\infty (\mu, \mathcal{B}_{n+k})$ be a uniformly bounded sequence.

Then, with probability one
$$\lim_{N\rightarrow \infty} \frac{1}{N} \sum_{n=1} ^N \left( f_n - \mathbb{E}(f_n | \mathcal{B}_n ) \right) =0.$$
\end{theorem}
The case when $k=1$ is the standard version \cite[Chapter 7, Theorem 3]{feller1991introduction}. The case of general $k\in \mathbb{N}$ follows from it by a short argument \cite[Section 2.1]{Hochman2020Host}.

$$ $$
\noindent{\textbf{Proof of Theorem \ref{Theorem Martingale}}} Let $f\in \mathcal{F} \subseteq C(X)$, where $\mathcal{F}$ is a dense and countable set. It suffices to show that for $\mu$-a.e. $x$
\begin{equation} \label{Eq (2.3)}
\lim_{N \rightarrow \infty} \left( \frac{1}{N} \sum_{n=1} ^N f(T ^n (x))  -  \frac{1}{N} \sum_{n=1} ^N \int f\,d T ^n \mu_{\mathcal{A}_n(x)} \right)=0.
\end{equation}
By \eqref{Eq 2.2},
$$\lim_{k\rightarrow \infty} f\circ T^n(x)- \mathbb{E} \left( f\circ T^n | \mathcal{A}_{n+k} \right)=0.$$
Moreover, this limit is uniform in $n\in \mathbb{N}$ and $x\in \text{supp}(\mu)$. So, it is enough to establish \eqref{Eq (2.3)} with $\mathbb{E} \left( f\circ T^n | \mathcal{A}_{n+k} \right)$ instead of $f\circ T^n$. As for the other term, for all $n\in \mathbb{N}$ we have
$$ \int f\,d T ^n \mu_{\mathcal{A}_n(x)} =  \int f\circ  T ^n \,d  \mu_{\mathcal{A}_n(x)} = \mathbb{E}_\mu \left( f\circ T^n | \mathcal{A}_n \right)(x).$$
Having made these adjustments, \eqref{Eq (2.3)} is now a direct consequence of Theorem \ref{Theorem erogidc diff} applied to the functions
$$f_n =    \mathbb{E}_\mu \left( f\circ T^n | \mathcal{A}_n \right).$$
The Theorem is proved. \hfill{$\Box$}


\section{Proof of Theorem \ref{main tech theorem 1}} \label{Section proof}
In this Section we prove Theorem \ref{main tech theorem 1}. Let us first fix some definitions and notation that will accompany us throughout the proof.  Let $\nu=\nu_\mathbf{p}$ be a self-similar measure with respect to a probability vector $\mathbf{p}$ and an IFS $\Phi = \lbrace f_1,...,f_n\rbrace$, where $f_i(x)=s_i\cdot x+t_i$ are invertible affine real contractions. Without the loss of generality we may assume that $\Phi$  preserves the interval $J=[0,1]$; in particular, $\nu\in \mathcal{P}([0,1])$.

 There is a natural coding that comes with a self-similar measure such as $\nu$: Let $\mathcal{A}:=\lbrace 1,...,n\rbrace$. For every $\omega \in \mathcal{A}^\mathbb{N}$ and $m\in \mathbb{N}$ define
$$f_{\omega|_m}:=f_{\omega_1}\circ \dots \circ f_{\omega_n}.$$
Fix $x_0 \in [0,1]$. Let $\Pi:\mathcal{A}^\mathbb{N} \rightarrow K$ be the  coding map
$$\Pi(\omega)=x_\omega:=\lim_{n\rightarrow \infty} f_{\omega|_m}(x_0).$$
Let $\rho = \max\lbrace |s_1|,...,|s_n|\rbrace$ and define a metric on $\mathcal{A}^\mathbb{N}$ via
$$d(\omega,\omega')=\rho^{ \min\lbrace n: \omega_n\neq \omega_n'\rbrace}.$$
We put the Bernoulli measure $\mathbb{P}=\mathbf{p}^\mathbb{N}$ on $\mathcal{A}^\mathbb{N}$. Then clearly we have
$$\nu = \Pi \mathbb{P}.$$

Next, let $\pi:\mathbb{R}\rightarrow \mathbb{R}/\mathbb{Z}\simeq [0,1)$ be the map
$$\pi(x)=x\mod 1.$$
Let $p\in \mathbb{N}$ be a positive integer, $p>1$.  Recall that $T_p:\mathbb{R} \rightarrow [0,1)$ is the map
$$T_p(x)=p\cdot x \mod 1.$$
Similarly, we define the map $\hat{T}_p :\mathbb{T} \rightarrow \mathbb{T}$ as
$$\hat{T}_p (x)=p\cdot x \mod 1.$$
We then have the following relation between $T_p$ and $\hat{T}_p$:
\begin{equation} \label{Eq (47)}
\hat{T}_p ^n \circ \pi(x)= \pi \circ T_p ^n (x) = p^n \cdot x \mod 1,\, \text{ for all } x\in \mathbb{R},\, n\in \mathbb{N}.
\end{equation}
Note that in \eqref{Eq (47)}  we use that $p$ is an integer. 
\subsection{A criterion for pointwise normality of self-similar measures}
We retain the notations set up at the beginning of this Section. In particular, 
recall that $\nu = \Pi \mathbb{P}$, and that $\mathcal{F}_q (\mu)$ is the Fourier mode of the measure $\mu$ at frequency $q$. We are now ready to preform our first reduction - a criterion for pointwise $p$-normality of $\nu$:
\begin{lemma} \label{Theorem 5.2}
Let  $\nu\in \mathcal{P}(\mathbb{R})$ be a  self-similar measure, and let $p>1$ be an integer. Suppose that:\newline
For every $\epsilon>0$ there exists $Q=Q(\epsilon)\in \mathbb{N}$ such that for every integer $q$ with $|q|>Q$ and for $\mathbb{P}$-a.e. $\omega$,
$$\limsup_{N\rightarrow \infty} \left| \mathcal{F}_q \left( \frac{1}{N} \sum_{n=1} ^N \delta_{T_p ^n (x_\omega)} \right) \right| <\epsilon.$$
Then $\nu$-a.e. $x$ is normal to base $p$.
\end{lemma}


\begin{proof}
By our assumption, the set
$$\lbrace  \omega \in \mathcal{A}^\mathbb{N}:\, \forall k>0\,\exists Q=Q(k)\, \forall q\in \mathbb{Z},|q|>Q,\, \limsup_{N\rightarrow \infty} \left| \mathcal{F}_q \left( \frac{1}{N} \sum_{n=1} ^N \delta_{T_p ^n (x_\omega)} \right) \right| <\frac{1}{k} \rbrace$$
carries full $\mathbb{P}$-measure. Let then fix an $\omega$ in this set. We will show that, in the weak$^*$ topology,
$$\lim_{N\rightarrow \infty} \frac{1}{N} \sum_{n=1} ^N \delta_{T_p ^n (x_\omega)} = \lambda,$$
where we recall that $\lambda$ is the Lebesgue measure on $[0,1]$.  By compactness, it suffices to show that every accumulation point (i.e. a measure on $[0,1)$) of the left hand side of the equation above is equal to $\lambda$. Without the loss of generality, let assume the limit above exists and equals a measure $\mu$.

To show that $\mu = \lambda$, it suffices to show that $\mathcal{F}_q(\mu)=0$ for all $q\in \mathbb{Z},q\neq 0$. By \eqref{Eq (47)} and the continuity of the map $\pi$, the measure $\pi \mu \in \mathcal{P}(\mathbb{\mathbb{T}})$ satisfies that
\begin{equation} \label{Eq port}
\pi \mu = \lim_{N\rightarrow \infty} \frac{1}{N} \sum_{n=1} ^N \delta_{\hat{T}_p ^n (\pi(x_\omega))}.
\end{equation}
The equation above implies that $\pi \mu$ is $\hat{T}_p$-invariant, i.e. 
$$\hat{T}_p  \pi \mu = \pi \mu.$$

Now, let $k\in \mathbb{N}$. Choose $n\in \mathbb{N}$ sufficiently large so that
$$\left| q\cdot p^n \right| \geq Q(k).$$
Then by the choice of $\omega$,
\begin{eqnarray*}
\left| \mathcal{F}_q (\mu) \right| &=&\left| \mathcal{F}_q ( \pi \mu) \right|   \\
&=&\left| \mathcal{F}_q ( \hat{T}_p ^n \pi \mu) \right|\\
&=&\left| \mathcal{F}_{qp^n} ( \pi \mu) \right|\\
&\leq & \limsup_{N} \left| \mathcal{F}_{qp^n} \left( \pi \left(  \frac{1}{N} \sum_{n=1} ^N \delta_{T_p ^n (x_\omega)} \right) \right) \right| \\
&= & \limsup_{N} \left| \mathcal{F}_{qp^n}  \left(  \frac{1}{N} \sum_{n=1} ^N \delta_{T_p ^n (x_\omega)} \right)  \right| \\
&<& \frac{1}{k}.
\end{eqnarray*}
Indeed, for the first equation we make use of the definition of $\pi$ and that $q\in \mathbb{Z}$, the second follows from $\hat{T}_p$-invariance, the third again uses that both $p,q\in \mathbb{Z}$, the fourth inequality follows from the Portmanteau theorem and \eqref{Eq port}, the next one uses the definition of $\pi$, and the last one our choice of $n$. Since $k$ was arbitrary, we conclude that $ \mathcal{F}_q (\mu)=0$, which implies the Lemma.
\end{proof}

\subsection{Applying Theorem \ref{Theorem Martingale}}
In this Section we apply Theorem \ref{Theorem Martingale}.

Let $p\geq 2$ be an integer. For every $n\in \mathbb{N}$ we define a stopping time $\beta_n: \mathcal{A}^\mathbb{N} \rightarrow \mathbb{N}$ via
$$\beta_n(\omega):=\min \lbrace m:\, \left| f_{\omega|_m} '(0) \right| <e^{-n\log p}\rbrace.$$
Since $\Phi$ comprises of affine maps, it is not hard to see that there exists some uniform $c=c(p)>0$ such that 
\begin{equation} \label{min stopping time}
\beta_n(\omega) \geq c\cdot n
\end{equation}
Recall the notations $\pi,\Pi$ from the previous sections.
\begin{lemma} \label{Theorem 5.5}
For $\mathbb{P}$-a.e. $\omega$, for every  $q\in \mathbb{Z}$ we have
$$\lim_{N\rightarrow \infty} \mathcal{F}_q \left( \frac{1}{N} \sum_{n=0} ^{N-1} \delta_{T_p ^n (x_\omega)} \right) - \mathcal{F}_q \left( \frac{1}{N} \sum_{n=0} ^{N-1} T_p ^n \circ f_{\omega|_{\beta_n(\omega)}} \nu \right)=0.$$ 
\end{lemma}
\begin{proof}
We first set the stage for the application of Theorem \ref{Theorem Martingale}. 
Consider the compact space $X= \mathcal{A}^\mathbb{N} \times \mathbb{T}$. We put the metric $d_\rho$ on the first coordinate, the usual metric on the second one, and the $\sup$-metric on the product space $X$. Next, for every $n$ consider the partition $\mathcal{C}_n$ of $X$ given by:
$$\mathcal{C}_n \left( \omega, \pi (x_\omega) \right) = \lbrace (\eta, \pi(x_\eta)):\, (\omega_1,...,\omega_{\beta_n(\omega)}) = (\eta_1,...,\eta_{\beta_n(\eta)} ) \rbrace,$$
$$B:=\mathcal{C}_n (\omega, x) = \lbrace (\eta,y):\, y\neq \pi(x_\eta) \rbrace \text{ for } x\neq \pi(x_\omega).$$
That is, we are grouping together elements of $X$ via the stopping time $\beta_n$ applied over the graph of the map $\pi \circ \Pi:\mathcal{A}^\mathbb{N}\rightarrow \mathbb{T}$; all elements outside of the graph are grouped together into a single cell called $B$, which will have negligible effect on the proof. Next, we define the probability measure $\mu \in \mathcal{P}(X)$ via
$$\mu(A)= \mathbb{P} \left( \lbrace \omega: \, (\omega, \pi(x_\omega) \in A \rbrace \right).$$
In particular,
$$\text{the projection of } \mu \text{ to } \mathbb{T} \text{ is } \pi\nu, \text{ and } \mu(B)=0.$$
Finally, we define $T:X\rightarrow X$ to be the continuous map
$$T(\omega,x):=\left( \omega,\, \hat{T}_p(x) \right).$$

To apply Theorem \ref{Theorem Martingale} we need to check that condition \eqref{Eq (2.3)} is met. To this end, let $n,k\in \mathbb{N}$ and fix some $A\in \mathcal{C}_{n+k}$ such that $\mu(A)>0$. Let $\omega \in \mathcal{A}^\mathbb{N}$ be such that $A=\mathcal{C}_n(\omega,x_\omega)$.  Then, by the definitions of $T$ and the metric $d_\rho$, and by \eqref{min stopping time},
$$\text{the diameter of the projection of } T^n A \text{ to the first coordinate is } \rho^{ \beta_{n+k} (\omega)} \leq \rho^{c\cdot (n+k)}.$$
In addition, recalling that $\text{supp}(\nu)=K$, we have
$$\text{the diameter of the projection of } T^n A \text{ to the first coordinate is } \text{diam}\left( T_p ^n \circ \pi \circ f_{\omega|_{\beta_{n+k}(\omega)}}(K) \right) = O(p^{-k}).$$
Taking tally of these two previous estimates and recalling the definition of the norm on $X$, we arrive at
$$ \text{diam}\left( T^n A \right) = O\left( \max \left\lbrace \rho^{c\cdot (n+k)}, p^{-k} \right\rbrace \right).$$
This converges to $0$ as $k\rightarrow \infty$ uniformly in $n$ and $\omega$.

Thus, Condition \eqref{Eq (2.3)} is met, and we may apply Theorem \ref{Theorem Martingale}: For $(\omega, \pi(x_\omega))\sim \mu$, writing $\mu_{\mathcal{C}_n(\omega,x_\omega)}$ for the conditional measure of $\mu$ on $\mathcal{C}_n(\omega,x_\omega)$, we have
$$\lim_{N\rightarrow \infty} \left( \frac{1}{N} \sum_{n=0} ^{N-1} \delta_{T ^n (\omega, x_\omega)}  -\frac{1}{N} \sum_{n=0} ^{N-1} T ^n \mu_{\mathcal{C}_n(\omega,x_\omega)} \right) =0.$$
Let us now project the equation above to $\mathbb{T}$: Since the projection of $\mu_{\mathcal{C}_n(\omega,x_\omega)}$ to $\mathbb{T}$ is precisely $\pi\circ f_{\omega|_{\beta_n(\omega)}} \nu$, we see that for $\mathbb{P}$-a.e. $\omega$
$$\lim_{N\rightarrow \infty} \left( \frac{1}{N} \sum_{n=0} ^{N-1} \delta_{\hat{T} _p ^n \circ \pi (x_\omega)}  -\frac{1}{N} \sum_{n=0} ^{N-1} \hat{T} ^n _p \circ \pi\circ f_{\omega|_{\beta_n(\omega)}} \nu \right) =0.$$

Finally, let $q\in \mathbb{Z}$, and apply the Fourier mode $\mathcal{F}_q$ to the equation above. Then, by \eqref{Eq (47)}
$$\mathcal{F}_q \left( \delta_{\hat{T} _p ^n \circ \pi (x_\omega)} \right) = \mathcal{F}_q \left( \delta_{T_p ^n (x_\omega)} \right)\, \text{ and } \mathcal{F}_q \left( \hat{T} ^n _p \circ \pi\circ f_{\omega|_{\beta_n(\omega)}} \nu \right) = \mathcal{F}_q \left( T^n _p \circ  f_{\omega|_{\beta_n(\omega)}} \nu \right).$$
The previous two displayed equations yield the Lemma. 
\end{proof}

\subsection{Conclusion of proof}
We are now in position to prove Theorem \ref{Main technical Theorem 2}. Let $p\in \mathbb{N},p\geq 2$.  By Lemma \ref{Theorem 5.2}, it suffices to show that: \newline
For every $\epsilon>0$ there exists $Q=Q(\epsilon)\in \mathbb{N}$ such that for every integer $q$ such that $|q|>Q$ and for $\mathbb{P}$-a.e. $\omega$,
$$\limsup_{N\rightarrow \infty} \left| \mathcal{F}_q \left( \frac{1}{N} \sum_{n=1} ^N \delta_{T_p ^n (x_\omega)} \right) \right| <\epsilon.$$
By Lemma \ref{Theorem 5.5}, this will follow if for every $\epsilon>0$ there exists $Q=Q(\epsilon)\in \mathbb{N}$ such that for every integer $q$ such that $|q|>Q$ and for $\mathbb{P}$-a.e. $\omega$, and for every $n\in \mathbb{N}$ we have
\begin{equation} \label{Eq need}
\left| \mathcal{F}_q \left( T_p ^n \circ f_{\omega|_{\beta_n(\omega)}} \nu \right) \right| <\epsilon.
\end{equation}
This is what we will prove.

Let us fix such an integer $q$, $\epsilon>0$, and a $\mathbb{P}$-typical $\omega$. Let $n\in \mathbb{N}$ and $x\in K$. Then since $\Phi$ is made up of affine maps (that are defined on all of $\mathbb{R}$), we can write:
\begin{equation} \label{5.5}
T^n _p \circ  f_{\omega|_{\beta_n(\omega)}} (x) = p^n \left( f_{\omega|_{\beta_n(\omega)}}' (0)\cdot x+f_{\omega|_{\beta_n(\omega)}}(0) \right) -m_{x,n},\, \text{ where } m_{x,n}\in \mathbb{Z}.
\end{equation}
Next, similarity to e.g. \eqref{min stopping time}, there is some $C_0=C_0(p)$ such that for all $n\in \mathbb{N},\omega \in \mathcal{A}^\mathbb{N}$,
$$\left| p^n \cdot  f_{\omega|_{\beta_n(\omega)}}' (0) \right|\in [C_0,1].$$

Now, write $r(\omega,n):=p^n \cdot  f_{\omega|_{\beta_n(\omega)}}' (0)$. Then, by \eqref{5.5} and since $q\in \mathbb{Z}$
$$\left| \mathcal{F}_q \left( T_p ^n \circ f_{\omega|_{\beta_n(\omega)}} \nu \right) \right| = \left| \mathcal{F}_{q\cdot r(\omega,n)} \left( \nu \right) \right|.$$
Since $\nu$ is a Rajchman measure, there exists some $Q'=Q'(\epsilon)$ such that for all $|q|>Q$,
$$\left| \mathcal{F}_{q} \left( \nu \right) \right|<\epsilon.$$
Since $r(\omega,n) \geq C_0$, if $Q=Q\cdot C_0$, then taking $|q|>Q$ we must have $|q\cdot r(\omega,n)|>Q'$, whence
$$\left| \mathcal{F}_{q\cdot r(\omega,n)} \left( \nu \right) \right|<\epsilon.$$
The previous three equations thus prove \eqref{Eq need}. The proof is complete.

\section{Some open problems} \label{Section problems}
In this Section we discuss a few open problems motivated by Theorems \ref{Theorem main},  \ref{main tech theorem 1}, and \ref{Main technical Theorem 2}.
\subsection{Effective equidistribution}
In this Section we consider the question of \textit{quantitative} (rather than qualitative) equidistribution of an orbit $\lbrace T_b ^n (x) \rbrace_{n\in \mathbb{N}}$, when $x$ is sampled from a self-similar measure:
\begin{problem} \label{Question effective}  Let $\nu$ be a self similar measure with respect to the IFS $\lbrace f_i(x)=s_i \cdot x+t_i \rbrace$. Suppose that $m$ is a positive integer such that  $m \not \sim s_i$ for some $i$. By Theorem \ref{Main technical Theorem 2}, we know that for $\nu$-a.e. $x$ we have
\begin{equation}
\frac{1}{N} \sum_{k=0} ^{N-1} f(T_m ^k x)= \int\, f(z)d\lambda(z) + o(1),\, \quad \, \forall f\in \mathcal{C}(\mathbb{T}).
\end{equation}
Can one specify a rate of convergence here? That is, can we say something about the $o(1)$ appearing above?
\end{problem}
Note that the argument given in Section \ref{Section proof} essentially  reduces the problem of  showing (qualitative) $\nu$-pointwise $p$ normality to that of showing 
$$\lim_{|q|\rightarrow \infty,\, q\in \mathbb{Z}} \mathcal{F}_q \left( T_p ^n \circ f_{\omega|_{\beta_{n}(\omega)}} \nu \right) =0, \text{ uniformly in } n, \omega.$$ 
There are several instances where  effective decay rates have been established in the limit above: Polynomial decay (of magnitude $|q|^{-\alpha},\alpha>0$) is known in some very few special cases \cite{streck2023absolute, Dai2007feng}, though it is known to hold in a strong generic sense for most self-similar measures \cite{Solomyak2021ssdecay};  Logarithmic decay (of  magnitude $1/\log \left( |q|\right)^\alpha ,\alpha>0$) has been established in a wide array of examples where the contraction ratios $\lbrace s_i \rbrace$ satisfy some rather mild   Diophantine conditions - see e.g. the very recent works \cite{Algom2021decay, li2019trigonometric, varju2020fourier} and references therein.

However,  the ergodic theorem for martingale differences (Theorem \ref{Theorem erogidc diff}) that is used to prove Theorem \ref{Theorem Martingale},  is not an effective result in general. Thus, to answer Problem \ref{Question effective}, one would need to find a different route towards equidistribution.

We remark that Problem \ref{Question effective} is also interesting when the underlying measure is a self-conformal measure, that is, when the underlying IFS is made up of $C^1$ invertible contractions of some interval. In this setting, a   version of Theorem \ref{Main technical Theorem 2} was given in \cite{barany2023scaling}, and a  version of Theorem \ref{main tech theorem 1} was given in \cite{algom2020decay}. Moreover, assuming the IFS is $C^2$, it had been shown in \cite{Baker2023sahl} and \cite{algom2023polynomial} independently that the transfer operator corresponding to the derivative cocycle and a given self-conformal measure \cite[Definition 2.3]{algom2023polynomial} has spectral gap \cite[Theorem 2.8]{algom2023polynomial}. As this property is used to prove polynomial Fourier decay for self-conformal measures, it is interesting to study if it can also be used to answer Problem \ref{Question effective} directly (without following the steps outlined in Section \ref{Section proof}).  

Finally, we remark that in the context of the $\times 2, \times 3$ conjecture, there is a related paper of  Bourgain, Lindenstrauss, Michel, and Venkatesh \cite{Bourgain2009Lind}. It is  an effective version of Rudolph's Theorem (the positive entropy case of the conjecture). Furthermore, in \cite{Bourgain2009Lind} one of the arguments  relies on Host's approach \cite{host1995normal} -  that is essentially Theorem \ref{Main technical Theorem 2} but for $T_p$-invariant measures.  

\subsection{The non-integer case}
In this section we consider the more general case of $\beta$-transformations: For $\beta>1$ we consider the map 
$$T_\beta(x)= \beta\cdot  x \mod 1,$$ 
similarly to the integer case considered previously in the paper. It is known that $T_\beta$ admits a unique absolutely continuous invariant measure, commonly referred to as the Parry measure. When $\beta$ is an integer then the Parry measure is just $\lambda$. Answering a question posed by Hochman and Shmerkin \cite{hochmanshmerkin2015} and confirming a conjecture of Bertrand-Mathis, very recently Huang and Wang \cite{huang2025coincidence}  have completely classified when two  different non-integers have the same Parry measure.

Thus, in analogy with the integer case, we say that $x$ is $\beta$-normal if it equidistributes under $T_\beta$ for the Parry measure. Now, it had been noted by Hochman-Shmerkin in \cite{hochmanshmerkin2015}  that their extension of Host's Theorem to self-similar measures  applies for pointwise normality in Pisot bases as well; In fact, Theorem \ref{Main technical Theorem 2} holds as stated if one substitutes any Pisot number $\beta>1$ instead of $b$, which was proved in this generality in \cite{Algom2022Baker} and in \cite{barany2023scaling}.

On the other hand, our method from \cite{algom2020decay}  outlined in Section \ref{Section proof} is based on Fourier analysis,  and is thus less suitable to treat $\beta$-normality for non-integer $\beta$. One key issue is that the identity $T_p ^n = T_{p^n}$, which is used  several times in our method (e.g. in \eqref{Eq (47)}) and is trivial for integers $p\geq 2$, is not true for $T_\beta$ when $\beta$ is not an integer.

We thus pose the following question:
\begin{problem} \label{Host beyond Pisot} (Pointwise $\beta$-normality beyond Pisot numbers) Let $\nu$ be a non-atomic self-similar measure with respect to the self-similar IFS $\lbrace f_i(x)=s_i \cdot x+t_i \rbrace$. Let $\alpha>1$ be any real number. Is it true that if for some $i$ we have $s_i \not \sim \alpha$ then $\nu$ is pointwise $\alpha$-normal? 
\end{problem}
We remark that one may consider this problem also for $T_b$-invariant measures, where the assumption here would be that $b \not \sim \alpha$. Recently, Fishbein \cite{Nevo2024some} had made some progress in this direction.

In addition, based on our previous discussion and Theorem \ref{Main technical Theorem 2}, the following question also arises naturally:
\begin{problem} (Rajchman self-similar measures are normal in non-integer Pisot bases) \label{Question rajchman}  Let $\nu$ be a non-atomic self-similar measure. If $\nu$ is a Rajchman measure, is it pointwise normal with respect to any Pisot number $\alpha>1$?
\end{problem}
We remark that Problem \ref{Question rajchman} can also be considered for non-Pisot numbers. However, the following example, indicated to me by Meng Wu, shows that even the precise formulation in this setting is not entirely clear. For every $\beta >2$  consider the   measure $\nu_\beta$, which is the self-similar measure corresponding to the IFS 
$$\lbrace f_1(x)=x/\beta,\,f_2(x)=(x+1)/\beta \rbrace$$ 
and the probability vector $\mathbf{p}=(\frac{1}{2}, \frac{1}{2})$. Then $\nu_\beta$ is supported on the self-similar set
$$K_\beta = \left\lbrace \sum_{n=1} ^\infty \frac{x_n}{\beta^n}:\,x_n \in \lbrace 0,1\rbrace \right \rbrace.$$
The set $K_\beta$ is clearly $T_\beta$ invariant; so, no $x\in K_\beta$ can be $\beta$-normal.

On the other hand, a suitable variant of the Erdős-Kahane argument  \cite[Section 6]{Peres200sixty} shows that for many $\beta>2$ the measure $\nu_\beta$ will have polynomial Fourier decay. This shows that Fourier decay does not, in general, force any element in the support of a measure to be $T_\beta$ normal.

 Finally, we note that by \cite{bremont2019rajchman}, if $\beta$ happens to be a Pisot number then $\nu_\beta$ is not a Rajchman measure, whence this family of measures contains no counter example towards Problem \ref{Question rajchman}. 
\subsection{Higher order correlations}
A sequence $x_n \in \mathbb{T}$  is called uniformly distributed or equidistributed if for any interval $J\subseteq \mathbb{T}$, 
\begin{equation*}
    \lim_{N\rightarrow \infty}\frac{1}{N}\sharp\left\lbrace 1\leq n\leq N:\, 
    x_n \in J\right\rbrace= \lambda(J).
\end{equation*}
Thus, $x\in \mathbb{R}$ is normal to base $b$ if $\lbrace T_b ^n (x) \rbrace_{n\in \mathbb{N}}$ is uniformly distributed. This notion  concerns the proportion  of the sequence that lies in a given interval. The fine-scale statistics of sequences in $\mathbb{T}$, which describe the behaviour of a sequence on the scale of mean gap $1/N$, have been attracting growing attention in recent years. We describe two such notions here:  $k$-level  correlations, and the distribution of level spacings (nearest-neighbour gaps), which are defined as follows.

For every  $N\in \mathbb{N}$ and $k\geq 2$ define,
$$\mathcal{U}_k=\left\{\textbf{u}=(u_1,\ldots,u_k):\, u_i\in\{1,\ldots,N\},\, u_i\neq u_j\, \text{ for all }  i\neq j\right\} \subseteq \mathbb{N}^k.$$ 
Fix $\textbf{u}\in\mathcal{U}_k$ and a sequence $x_n\in \mathbb{R}$, and consider the difference vector
$$\Delta(\textbf{u}, x_n)=(x_{u_1}-x_{u_2},\ldots,x_{u_{k-1}}-x_{u_k})\in\mathbb{R}^{k-1}.$$ 
Let $f:\mathbb{R}^{k-1} \rightarrow \mathbb{R}$ be a compactly supported function. We define the $k$-level correlation function to be
\begin{equation}\label{k level correlation}
    R_k(f, x_n, N):=\frac{1}{N}\sum_{\textbf{u}\in\mathcal{U}_k}\sum_{\textbf{l}\in\mathbb{Z}^{k-1}}f(N(\Delta(\textbf{u}, x_n)+\textbf{l})).
\end{equation} 
If the following limit exists: 
\begin{equation}\label{poisson limit}
    \lim_{N\rightarrow\infty} R_k(f, x_n, N)=\int_{\mathbb{R}^{k-1}}f(\textbf{x})  d\textbf{x},\; \forall f\in C^{\infty}_c(\mathbb{R}^{k-1})
\end{equation}
then we say that the $k$-level correlation of $x_n$ is Poissonian. This notion hints at the fact that such behaviour is consistent with the almost sure behaviour of a Poisson process with intensity one.

The distribution of level spacings (also called nearest neighbour gaps) of the  $x_n$ describe the gaps between successive  elements of $x_n$. For  $N \in \mathbb{N}$, we order the first $N$ elements of the sequence $x_n$ and re-label them as 
\begin{equation*}
   0\leq\theta_{1, N}\leq\theta_{2, N}\leq\dots\leq\theta_{N, N}\leq 1,\quad \text{ and set } \theta_{0, N}:=\theta_{N, N}-1 \pmod 1.
\end{equation*}
Let us assume that for every $s\geq 0$ the limit as $N\rightarrow \infty$ of the function
\begin{equation*}
\lim_{N\rightarrow \infty}    G(s, x_n, N):=\frac{1}{N}\sharp\left\{1\leq n\leq N:\, N\left( \theta_{n, N}-\theta_{n-1, N} \right)\leq s\right\} \text{ exists.}
\end{equation*}
The limit function $G(s)$ is then called the asymptotic distribution function of the level spacings of $x_n$. The level spacings are called Poissonian if 
$$G(s)=1-e^{-s}.$$ 
Note that this is  the distribution of the waiting time of a Poisson process with intensity one. These notions are standard ways of describing the pseudo-randomness properties of sequences - see  \cite{Yesha2023Baker} for more details.

We can now state our final Problem:
\begin{problem} (higher order correlations) \label{Question higher order corr}
Let $b>1$ be an integer, and let $\nu \in \mathcal{P}(\mathbb{R})$ be a non-atomic self-similar measure with respect to an IFS $\Phi$. Suppose that there exists $f_i \in \Phi$ such that $f_i ' \not \sim b$.
\begin{enumerate}
\item  Let $k \geq 2$. Are the $k$-level correlation of $T_b ^n x$  Poissonian for $\nu$-a.e. $x$?

\item Are the level spacings Poissonian?
\end{enumerate}

\end{problem}
There are some recent results in the literature that study related problems. However, this is mostly for orbits of typical points under \textit{non-linear} maps.  A classical Theorem of Koksma states that $\lbrace x^n \mod 1 \rbrace_{n\in \mathbb{N}}$ is  equidistributed for Lebesgue a.e. $x>1$. Aistleitner  and Baker \cite{Aistleitner2021Baker} recently strengthened this results by   proving that $\lbrace x^n \mod 1 \rbrace_{n\in \mathbb{N}}$ has Poissonian pair correlations for almost every $x>1$ (Poissonian pair correlations is known to  imply uniform distribution \cite{Chris2018Lach, Larcher2017equi, Jens2020equi}).  More recently, Aistleitner, Baker, Technau, and Yesha \cite{Yesha2023Baker} proved that typically such sequences  have Poissonian $k$-level correlations, and thus obtained that the level spacings of $\lbrace x^n \mod 1 \rbrace_{n\in \mathbb{N}}$ are Poissonian for almost surely (it is well-known \cite[Appendix A]{Rudnick1999par} that if the  $k$-level correlation of a sequence is Poissonian for all $k\geq 2$, then the level spacings are also Poissonian).

The results in the previous paragraph deal with $x>1$ that is sampled from the Lebesgue measure. Recall that such measures are always self-similar, but there exist many self-similar measures that are singular. Baker \cite{bakwer2021equi} showed that for some self-similar measures, $\lbrace x^n \mod 1 \rbrace_{n\in \mathbb{N}}$ is also uniformly distributed almost surely. This was shown, for example, for  translates of the Cantor-Lebesgue measure on the middle-thirds Cantor set. Baker  conjectured  that this should be true for every self-similar measure supported on $[1,\infty)$. The Conjecture  was very recently proved by the author, Chang, Meng Wu, and Yu-Liang Wu \cite{algom2023wu}. In fact, in \cite{algom2023wu} a much stronger result,  that $\lbrace x^n \mod 1 \rbrace_{n\in \mathbb{N}}$ typically has Poissonian $k$-level correlations for all $k\geq 2$, was proved.

\bibliography{bib}
\bibliographystyle{plain}

\end{document}